\documentclass{birkjour}
\usepackage{amsmath, amsthm, amssymb, hyperref}
\usepackage{graphicx, tikz}
\usepackage[section]{placeins}
\usetikzlibrary{fadings}
\tikzset{every picture/.style={line width=0.75pt}} 

\newtheorem{thm}{Theorem}[section]

\theoremstyle{definition}
\newtheorem{defn}{Definition}[section]
\numberwithin{equation}{section}

\DeclareMathOperator{\detl}{detl}
\DeclareMathOperator{\gdet}{gdet}
\DeclareMathOperator{\vdet}{vdet}

\DeclareMathOperator{\sgn}{sgn}

\begin{document}
\title[Generalised determinants: a review]{Generalisations of the determinant to interdimensional transformations: a review}
\author[Abhimanyu Pallavi Sudhir]{Abhimanyu Pallavi Sudhir}
\address{Department of Mathematics\\
Imperial College London\\
180 Queen's Gate, South Kensington Campus\\
London -- SW7 2AZ, UNITED KINGDOM}
\date{May 11, 2013}

\begin{abstract}
Significant research has been carried out in the past half-century on defining generalised determinants for transformations between (typically real) vector spaces of different dimensions. We review three different generalisations of the determinant to non-square matrices, that we term for convenience the determinant-like function \cite{Pasu1}, the vector determinant \cite{Pyle} and the g-determinant \cite{Radic1}. We introduce and motivate these generalisations, note certain formal similarities between them and discuss their known properties.
\end{abstract}

\maketitle

\section{Introduction}

Several generalisations of the determinant to domains consisting of non-square matrices are known, including the determinant-like function \cite{Pasu1}, the vector determinant \cite{Pyle}, and the g-determinant \cite{Radic1}. Each of these generalisations is defined based on different defining properties of the square determinant.

\begin{equation}\label{detl-def}
\detl (\vec{c}_1,\ldots \vec{c}_n)  \triangleq
\left|\bigwedge_j\vec{c}_j\right|
\end{equation}

Notably, the determinant-like function (Eq.~\ref{detl-def}) has perhaps the clearest geometric description (it is the volume of the image (in $\mathbb{R}^m$) of the unit $n$-cube under a linear transformation) -- however, it is also the most limited. The function is an \emph{unsigned} determinant, and is only defined for tall matrices (i.e. linear transformations to a higher-dimensional space). Both limitations arise directly from the geometric interpretation -- the function must be unsigned, as determining the handedness of a general $n$-volume in $\mathbb{R}^m$ depends on the introduction of additional vectors; and an extension of the geometric interpretation to wide matrices (transformations to a lower-dimensional space) would require the function to be uniformly zero.

Algebraically, the determinant-like function can be calculated as (we will derive this expression in Sec.~\ref{sec:detl}):

\begin{equation}\label{detl}
\detl A = \left[\sum_{i_1<\ldots i_n} {\det}^{2} A_{i_1,\ldots i_n}\right]^{1/2}
\end{equation}

Where $A_{i_1,\ldots i_n}$ are the ``maximal square submatrices'' of $A$, or the $n$ by $n$ submatrices of $A$. In fact, these submatrices also appear in the expressions for the vector determinant and the g-determinant. 

\begin{equation}\label{vdet}
\vdet A = \sum_{i_1<\ldots  i_n} \det A_{i_1,\ldots i_n} \vec{e}_{\tau(i_1,\ldots i_n)}
\end{equation}

The vector determinant Eq.~\eqref{vdet} is -- for some bijection $\tau$ to $\left[1,\binom{m}{n}\right]$ from the set of ascending $n$-tuples of integers $i_1,\ldots i_n$ ($1 \le i_j \le m$) -- a vector-valued function $\vdet: \mathbb{R}^{m\times n}\to \mathbb{R}^{\binom{m}{n}}$ for $m \ge n$. 

One might note that Eq.~\eqref{detl} is simply the magnitude of the vector in Eq.~\eqref{vdet} \cite{Pasu2} -- indeed, in this sense the vector determinant might be thought of as a way to ``mitigate'' the limitations of th determinant-like function. This relation will be discussed in further detail in Sec.~\ref{sec:vdet-disc}.

(Note that in the original reference \cite{Pyle}, the left-multiplication convention was adopted, i.e. an $m\times n$ matrix represented a transformation from $\mathbb{R}^m\to\mathbb{R}^n$, and thus the vector determinant was originally claimed to apply to ``wide'' rather than ``tall'' matrices -- however, we adopt the opposite convention for the sake of uniformity.)

The g-determinant \cite{Radic1} (a term that first appeared in \cite{Radic2}), defines the determinant of a tall matrix recursively through Laplace's expansion, with the determinant of a $m$ by 1 matrix as the base case:

\begin{equation}\label{gdet-base}
\gdet (a_1,\ldots a_m)^T = a_1 - a_2 + \ldots  + (-1)^{m + 1} a_m
\end{equation}

With a recursion relation given by the Laplace expansion:
\begin{equation}\label{gdet-ih}
\gdet A = \sum_{i = 1}^m (- 1)^{i + 1} a_{i1} \gdet A_{\{i\}^C}^{\{1\} ^C}
\end{equation}

Where $A_S^T$ refers to the submatrix of $A$ comprised of the rows in $S$ and the columns in $T$, and $S^C$ is the complement. It is then easy to show (and will be shown in Sec.~\ref{sec:gdet}) inductively that the general expression for the g-determinant of an $m$ by $n$ matrix ($m \ge n$) is given by:

\begin{equation}\label{gdet}
\gdet A = \sum_{i_1<\ldots i_n}(-1)^{\sum_j i_j + j}\det A_{i_1,\ldots i_n}
\end{equation}

\section{The determinant-like function}\label{sec:detl}

On the subject of determinants of non-square matrices, it is clear that an interpretation in terms of volume requires the determinant of a ``wide'' matrix (one with more columns than rows) to be zero, as it is a transformation from a higher-dimensional space into a lower-dimensional one. However, one may still assign a non-trivial volume-based interpretation to the determinant of a tall matrix -- despite being a linear map into a higher-dimensional space, such a transformation is not surjective (its rank is still at most the dimension of the domain).

\begin{figure}\label{fig:M1}
\begin{tikzpicture}[x=0.75pt,y=0.75pt,yscale=-0.70,xscale=0.75]

\draw  (207.39,241.4) -- (502.35,241.4)(344.84,144.2) -- (224.89,252.2) (500.9,236.4) -- (502.35,241.4) -- (489.8,246.4) (332.07,151.2) -- (344.84,144.2) -- (342.07,151.2)  ;
\draw    (236.3,259.2) -- (233.33,34.2) ;
\draw [shift={(233.3,32.2)}, rotate = 449.24] [color={rgb, 255:red, 0; green, 0; blue, 0 }  ][line width=0.75]    (10.93,-3.29) .. controls (6.95,-1.4) and (3.31,-0.3) .. (0,0) .. controls (3.31,0.3) and (6.95,1.4) .. (10.93,3.29)   ;

\draw [color={rgb, 255:red, 0; green, 0; blue, 255 }  ,draw opacity=1 ]   (236.89,241.4) -- (284.6,114.07) ;
\draw [shift={(285.3,112.2)}, rotate = 470.54] [color={rgb, 255:red, 0; green, 0; blue, 255 }  ,draw opacity=1 ][line width=0.75]    (10.93,-3.29) .. controls (6.95,-1.4) and (3.31,-0.3) .. (0,0) .. controls (3.31,0.3) and (6.95,1.4) .. (10.93,3.29)   ;

\draw  [draw opacity=0][fill={rgb, 255:red, 255; green, 0; blue, 0 }  ,fill opacity=0.25 ] (292.67,191.2) -- (376.86,191.2) -- (321.08,241.4) -- (236.89,241.4) -- cycle ;
\draw [color={rgb, 255:red, 255; green, 9; blue, 9 }  ,draw opacity=1 ][fill={rgb, 255:red, 255; green, 220; blue, 200 }  ,fill opacity=1 ]   (236.89,241.4) -- (291.81,192.53) ;
\draw [shift={(293.3,191.2)}, rotate = 498.33] [color={rgb, 255:red, 255; green, 9; blue, 9 }  ,draw opacity=1 ][line width=0.75]    (10.93,-3.29) .. controls (6.95,-1.4) and (3.31,-0.3) .. (0,0) .. controls (3.31,0.3) and (6.95,1.4) .. (10.93,3.29)   ;

\draw [color={rgb, 255:red, 255; green, 9; blue, 9 }  ,draw opacity=1 ][fill={rgb, 255:red, 255; green, 220; blue, 200 }  ,fill opacity=1 ]   (240.67,241.6) -- (319.08,241.4) ;
\draw [shift={(321.08,241.4)}, rotate = 539.86] [color={rgb, 255:red, 255; green, 9; blue, 9 }  ,draw opacity=1 ][line width=0.75]    (10.93,-3.29) .. controls (6.95,-1.4) and (3.31,-0.3) .. (0,0) .. controls (3.31,0.3) and (6.95,1.4) .. (10.93,3.29)   ;

\draw [color={rgb, 255:red, 255; green, 130; blue, 130 }  ,draw opacity=1 ][fill={rgb, 255:red, 255; green, 220; blue, 200 }  ,fill opacity=1 ]   (293.3,191.2) -- (376.3,192.2) ;

\draw [color={rgb, 255:red, 255; green, 130; blue, 130 }  ,draw opacity=1 ][fill={rgb, 255:red, 255; green, 220; blue, 200 }  ,fill opacity=1 ]   (321.08,241.4) -- (376.3,192.2) ;

\draw [color={rgb, 255:red, 0; green, 0; blue, 255 }  ,draw opacity=1 ]   (240.67,241.6) -- (398.13,158.45) ;
\draw [shift={(399.9,157.52)}, rotate = 512.1600000000001] [color={rgb, 255:red, 0; green, 0; blue, 255 }  ,draw opacity=1 ][line width=0.75]    (10.93,-3.29) .. controls (6.95,-1.4) and (3.31,-0.3) .. (0,0) .. controls (3.31,0.3) and (6.95,1.4) .. (10.93,3.29)   ;

\draw [color={rgb, 255:red, 130; green, 130; blue, 255 }  ,draw opacity=1 ][fill={rgb, 255:red, 255; green, 220; blue, 200 }  ,fill opacity=1 ]   (399.1,158.21) -- (448.34,27.14) ;

\draw [color={rgb, 255:red, 130; green, 130; blue, 255 }  ,draw opacity=1 ][fill={rgb, 255:red, 255; green, 220; blue, 200 }  ,fill opacity=1 ]   (285.3,112.2) -- (448.34,27.14) ;

\draw  [draw opacity=0][fill={rgb, 255:red, 0; green, 0; blue, 255 }  ,fill opacity=0.25 ] (285.21,112.78) -- (448.34,27.14) -- (399.1,158.21) -- (235.97,243.86) -- cycle ;

\draw (436,158) node [color={rgb, 255:red, 0; green, 0; blue, 255 }  ,opacity=1 ]  {$\vec{a}_{1} \in \mathbb{R}^3$};
\draw (283,91) node [color={rgb, 255:red, 0; green, 0; blue, 255 }  ,opacity=1 ]  {$\vec{a}_{2}$};
\draw (315,259) node [color={rgb, 255:red, 255; green, 0; blue, 0 }  ,opacity=1 ]  {$\vec{e}_{1} \in \mathbb{R}^2$};
\draw (282,182) node [color={rgb, 255:red, 255; green, 0; blue, 0 }  ,opacity=1 ]  {$\vec{e}_{2}$};
\end{tikzpicture}
\caption{A transformation $A:\mathbb{R}^2\to\mathbb{R}^3$}
\end{figure}
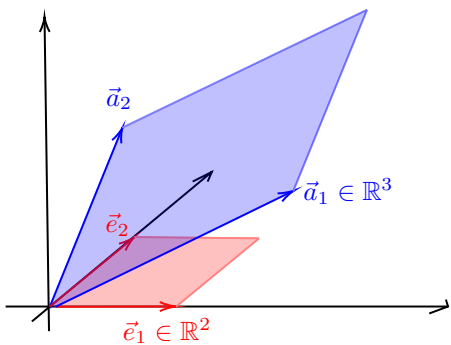

It is reasonable to define a determinant-like function for a tall matrix $A:\mathbb{R}^n\to\mathbb{R}^m$ as the $n$-volume of the image in $m$-dimensional space of the unit $n$-form (see Fig.~\ref{fig:M1} for an example). Formally, one may generalise the following exterior algebraic property of the standard determinant (for a matrix of $n$ vectors in $\mathbb{R}^n$):

\begin{equation}\label{exterior-prop}
    \left| \det (\vec{c}_1,\ldots \vec{c}_n) \right| = \left|\bigwedge_i\vec{c}_i\right|
\end{equation}

To the \emph{determinant-like function} $\detl$ that satisfies the following natural generalisation:

\begin{defn}\label{def:D1}
The determinant-like function of a matrix comprised of columns $\vec{c}_1,\ldots \vec{c}_n$ with each $c_j\in\mathbb{R}^m$ ($m \ge n$) is defined as the necessarily unsigned $n$-volume of the $n$-form spanned by these columns, i.e.
\begin{equation*}
    \detl (\vec{c}_1,\ldots \vec{c}_n)  \triangleq
    \left|\bigwedge_j\vec{c}_j\right|
\end{equation*}
\end{defn}

We will treat Definition~\ref{def:D1} as the definition of the determinant-like function for tall matrices, and proceed to find a computable, algebraic expression for the function.

\subsection{A computable expression for the determinant-like function}

With Definition~\ref{def:D1}, one may easily calculate the determinant-like function of elementary examples. Here, we produce a general usable expression for any tall matrix. In doing so, we correct the erroneous factor of $\sqrt{|m-n|!}$ included in the original reference \cite{Pasu1}.

\begin{thm}\label{thm:T1}
The unsigned determinant-like function of a tall $m$ by $n$ matrix $A$ is given by the following expression in terms of the $\binom{m}{n}$ $n$ by $n$ square submatrices of $A$ (here, $A_S$ where $S\subseteq\{1,\ldots m\}$, is the submatrix of $A$ whose rows are those indexed in $S$):
\begin{equation*}
\detl A = \left(\sum_{i_1<\ldots  i_n} {\det}^{2} A_{i_1,\ldots i_n}\right)^{1/2}
\end{equation*}
(Note that by convention, $1\le i_j \le m$ for $1 \le j \le n$.)
\end{thm}

\begin{proof}
We attempt to simplify the wedge product of the columns of $A$ by expressing them in the standard basis of $\mathbb{R}^m$ and applying the distributive law (note that $\mathbf{e}_{i_1,\ldots i_n}:=\bigwedge_j\vec{e}_{i_j}$):

\begin{align*}
\bigwedge_j\sum_i a_{ij}\vec{e}_i
& = \sum_{i_1,\ldots i_n} \bigwedge_j a_{i_j j} \vec{e}_{i_j}\\
& = \sum_{i_1,\ldots i_n} \prod_j a_{i_j j} \mathbf{e}_{i_1,\ldots i_n}
\end{align*}

Grouping the $\mathbf{e}_{i_1,\ldots i_n}$ terms with their index permutations $\sigma \in S_n$ and applying the antisymmetric nature of the wedge product, one can rewrite this as a sum over $\binom{m}{n}$ ascending combinations of $n$ indices.

\begin{align*}
\ldots &= \sum_{i_1<\ldots i_n}\sum_\sigma \prod_j a_{i_{\sigma(j)}j}\mathbf{e}_{i_{\sigma(1)},\ldots i_{\sigma(n)}} \\
& = \sum_{i_1<\ldots i_n}\mathbf{e}_{i_1,\ldots i_n}\sum_\sigma \sgn(\sigma)\prod_j a_{i_{\sigma(j)}j} \\
& = \sum_{i_1<\ldots i_n}\mathbf{e}_{i_1,\ldots i_n} \det A_{i_1,\ldots i_n}
\end{align*}

The result then follows from taking the norm of this $n$-form, which is definitionally $|X| = \left|\left\langle X, X\right\rangle \right|$.

\end{proof}

\subsection{Discussion}
We note two special cases of Theorem~\ref{thm:T1} immediately: the $m$ by 1 case is the Pythagoras theorem in $m$ dimensions, as it measures the 1-volume (length) of a vector in terms of its projections onto each of the $m$ dimensions. Also notably, the $(n+1)\times n$ determinant is the magnitude of the $n$-vector cross product, as it measures the $n$-volume contained between them (for example the 3 by 2 case is simply the elementary-school cross product).

It is worth commenting on the question of the sign of the determinant-like function -- in general in $m$ dimensions, one may naturally give a sign to an $m$-form by defining a left-handed/right-handed convention. To define a sign on an $n$-form, one needs a notion of orientation, however, in $m$ dimensions, this depends on which remaining $m - n$ vectors you use to define this orientation (e.g. in 3 dimensions, one needs a third vector to ``look at'' two vectors and determine which vector is clockwise to which). Fixing such vectors would not make sense for our purposes (e.g. consider a case where the $n$-form is parallel to one of our fixed vectors -- the determinant of the $m$-form comprising of the $n$-form and the fixed vectors is zero, neither positive nor negative, and this does not describe a sign on our $n$-volume at all).

We consider the validity of the corresponding generalisations of some standard properties of the square determinant to the determinant-like function -- unlike the square determinant, the manner in which the determinant-like function treats its columns is fundamentally different from the way that it treats its rows. 
:

\begin{itemize}
    \item An $m$ by $n$ matrix with any linearly independent columns has a zero determinant-like function (as the $n$-form ``closes in''), but it needs more than $m - n$ rows to be linearly dependent in order for the determinant-like function to be zero. 
    \item The determinant-like function is linear in each of its columns -- however, multiplying a row by a scalar $k$ has no such predictable effect.
    \item Adding a linear multiple of a column to another column leaves the determinant-like function invariant, but doing so with the columns has no such predictable effect.
    \item The function \emph{is} invariant under switching both columns and rows (as the component square determinants are just permuted).
    \item In general for matrices $A: \mathbb{R}^n \to \mathbb{R}^m$ and $B: \mathbb{R}^o \to \mathbb{R}^n$ ($m \ge n \ge o$), the determinant-like function is \emph{not} multiplicative, i.e. $\detl AB \ne \detl A \detl B$, as $\detl A$ describes the volume transformation of an $n$-form and not an $o$-form (indeed, $o$-forms would not be similarly scaled under $A$).
\end{itemize}

The generalisation of Cramer's rule to the determinant-like function was the subject of \cite{Pasu3}. Specifically, in a solvable rank-$n$ system of $m$ linear equations in $n$ variables ($m \ge n$, as usual), represented by $Ax=b$, the components of $x$ satisfy:

\begin{equation}\label{detl-sol}
|x_j| = \frac{\detl (\vec{c}_1,\ldots \vec{c}_j,\vec{b},\vec{c}_{j+1},\ldots \vec{c}_n)}{\detl A}
\end{equation}

This can be proven in the following straightforward way: for the system to have a solution, $b$ must be expressible as a linear combination of the linearly independent columns of $A$ -- therefore, any $n$ by $n$ submatrix of $A$ has determinant zero if and only if the corresponding square submatrix of the matrix ``$A$ with its $j$\textsuperscript{th} column replaced with $b$'' has determinant zero, and the ratio of any non-zero determinants of corresponding submatrices is exactly the solution to the system (as the system can be solved by eliminating linearly dependent equations). The result follows.

One might wonder if considering the values of the solutions themselves (as opposed to the absolute values) in Eq.~\eqref{detl-sol} would allow us to define a sign on the determinant-like function. Such a definition is possible, but not natural -- it would be equivalent to defining an orientation on every $n$-dimensional subspace of $\mathbb{R}^m$, each subspace's definition independent of one another (see Fig.~\ref{fig:M2}). Notably the resulting ``signed determinant-like function'' would not be continuous under any chosen convention.

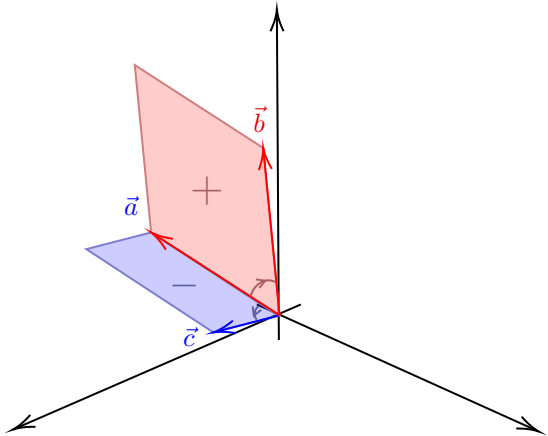
\begin{figure}\label{fig:M2}

\begin{tikzpicture}[x=0.75pt,y=0.75pt,yscale=-1,xscale=1]

\draw    (331.9,184.8) -- (330.91,19.8) ;
\draw [shift={(330.9,17.8)}, rotate = 449.66] [color={rgb, 255:red, 0; green, 0; blue, 0 }  ][line width=0.75]    (10.93,-3.29) .. controls (6.95,-1.4) and (3.31,-0.3) .. (0,0) .. controls (3.31,0.3) and (6.95,1.4) .. (10.93,3.29)   ;

\draw    (342.9,166.8) -- (200.73,229.24) ;
\draw [shift={(198.9,230.04)}, rotate = 336.28999999999996] [color={rgb, 255:red, 0; green, 0; blue, 0 }  ][line width=0.75]    (10.93,-3.29) .. controls (6.95,-1.4) and (3.31,-0.3) .. (0,0) .. controls (3.31,0.3) and (6.95,1.4) .. (10.93,3.29)   ;

\draw    (320.9,166.8) -- (460.08,230.21) ;
\draw [shift={(461.9,231.04)}, rotate = 204.49] [color={rgb, 255:red, 0; green, 0; blue, 0 }  ][line width=0.75]    (10.93,-3.29) .. controls (6.95,-1.4) and (3.31,-0.3) .. (0,0) .. controls (3.31,0.3) and (6.95,1.4) .. (10.93,3.29)   ;

\draw  [color={rgb, 255:red, 207; green, 125; blue, 125 }  ,draw opacity=1 ][fill={rgb, 255:red, 255; green, 0; blue, 0 }  ,fill opacity=0.2 ] (259.85,45.81) -- (324.08,87.72) -- (332.27,172.31) -- (268.03,130.41) -- cycle ;
\draw [color={rgb, 255:red, 255; green, 0; blue, 0 }  ,draw opacity=1 ]   (332.27,172.31) -- (324.27,89.71) ;
\draw [shift={(324.08,87.72)}, rotate = 444.47] [color={rgb, 255:red, 255; green, 0; blue, 0 }  ,draw opacity=1 ][line width=0.75]    (10.93,-3.29) .. controls (6.95,-1.4) and (3.31,-0.3) .. (0,0) .. controls (3.31,0.3) and (6.95,1.4) .. (10.93,3.29)   ;

\draw  [color={rgb, 255:red, 125; green, 125; blue, 207 }  ,draw opacity=1 ][fill={rgb, 255:red, 0; green, 0; blue, 255 }  ,fill opacity=0.2 ] (235.63,138.85) -- (268.03,130.41) -- (331.55,172.5) -- (299.14,180.94) -- cycle ;
\draw [color={rgb, 255:red, 0; green, 0; blue, 255 }  ,draw opacity=1 ]   (331.55,172.5) -- (301.08,180.43) ;
\draw [shift={(299.14,180.94)}, rotate = 345.4] [color={rgb, 255:red, 0; green, 0; blue, 255 }  ,draw opacity=1 ][line width=0.75]    (10.93,-3.29) .. controls (6.95,-1.4) and (3.31,-0.3) .. (0,0) .. controls (3.31,0.3) and (6.95,1.4) .. (10.93,3.29)   ;

\draw [color={rgb, 255:red, 255; green, 0; blue, 0 }  ,draw opacity=1 ]   (332.27,172.31) -- (269.71,131.5) ;
\draw [shift={(268.03,130.41)}, rotate = 393.12] [color={rgb, 255:red, 255; green, 0; blue, 0 }  ,draw opacity=1 ][line width=0.75]    (10.93,-3.29) .. controls (6.95,-1.4) and (3.31,-0.3) .. (0,0) .. controls (3.31,0.3) and (6.95,1.4) .. (10.93,3.29)   ;

\draw [color={rgb, 255:red, 168; green, 96; blue, 96 }  ,draw opacity=1 ]   (317.9,162.48) .. controls (318.9,157.48) and (324.9,152.48) .. (329.9,155.48) ;

\draw  [color={rgb, 255:red, 168; green, 96; blue, 96 }  ,draw opacity=1 ] (320.44,154.15) -- (325.07,154.81) -- (322.14,158.46) ;
\draw [color={rgb, 255:red, 95; green, 97; blue, 159 }  ,draw opacity=1 ]   (320.9,174.48) .. controls (318.9,172.48) and (319.9,168.48) .. (321.9,167.8) ;

\draw  [color={rgb, 255:red, 95; green, 97; blue, 159 }  ,draw opacity=1 ] (323.09,169.55) -- (319.41,172.44) -- (318.8,167.8) ;

\draw (258,117.04) node [color={rgb, 255:red, 0; green, 0; blue, 255 }  ,opacity=1 ] [align=left] {$\displaystyle \vec{a}$};
\draw (322,73.04) node [color={rgb, 255:red, 225; green, 0; blue, 0 }  ,opacity=1 ] [align=left] {$\displaystyle \vec{b}$};
\draw (287,183.04) node [color={rgb, 255:red, 0; green, 0; blue, 255 }  ,opacity=1 ] [align=left] {$\displaystyle \vec{c}$};
\draw (296.06,109.06) node [color={rgb, 255:red, 168; green, 96; blue, 96 }  ,opacity=1 ] [align=left] {{\huge +}};
\draw (284.59,158.67) node  [align=left] {{\huge \textbf{\textcolor[rgb]{0.37,0.38,0.62}{--}}}};

\end{tikzpicture}

\caption{An example orientation convention on the $xz$-plane}
\end{figure}

\section{The vector determinant}

The original historic motivation for the vector determinant Eq.~\eqref{vdet-rep} is not completely clear -- it may be described in a sense as an accidental discovery. When introduced in \cite{Pyle}, it was claimed that the vector determinant was definitionally the generalisation that satisfied the following properties:

\begin{itemize}
    \item Linearity in each column vector
    \item Irreflexivity, i.e. equals zero whenever two columns are equal
    \item The normalisation $\vdet(e_1, e_2 \ldots e_n)=(1,0,\ldots 0)$ or similar
\end{itemize}

\begin{equation}\label{vdet-rep}
\vdet A = \sum_{i_1<\ldots  i_n} \det A_{i_1,\ldots i_n} \vec{e}_{\tau(i_1,\ldots i_n)}
\end{equation}

However, no proof of uniqueness was presented, and in fact, a distinct generalisation satisfying these properties was presented later in \cite{Joshi}, and is recorded below in Eq.~\eqref{tdet}.

\begin{equation}\label{tdet}
\sum_{i_1<\ldots  i_n} \det A_{i_1,\ldots i_n}
\end{equation}

The key theorem relating to the vector determinant that makes it noteworthy has to do with determinants of products of the form $\det(A^T B)$ where $X$ and $Y$ have the same dimensions. We give the statement (and proof) of this theorem for real-valued matrices below.

\begin{thm}\label{thm:T2}
For two $m$ by $n$ matrices $A$ and $B$, the determinant of their square $n$ by $n$ inner product $A^T B$ is given by the dot product of their vector determinants:

\begin{equation*}
\det (A^T B) = (\vdet A)^T \vdet B
\end{equation*}

\end{thm}

\begin{proof}
Consider the $j$\textsuperscript{th} column of $A^TB$, $(A^TB)^j$. One may write this as a sum of $m$ vectors as follows, where $i$ runs from 1 to $m$:

\begin{equation*}
(A^T B)^j = \sum_i B_i^jA_i
\end{equation*}

The determinant of $A^T B$ is then the determinant of a matrix whose columns are sums of simpler columns -- therefore we may inductively apply the multilinearity of the determinant in these sums to write (where $1\le i_j\le m$):

\begin{align*}
\det {A^T}B 
& = \sum_{i_1, \ldots i_n } \det (B_{i_1}^{1} A_{i_1}, \ldots B_{i_n}^n A_{i_n}) \\
& = \sum_{i_1,\ldots i_n} B_{i_1}^1 \ldots B_{i_n}^n \det A_{i_1,\ldots i_n}
\end{align*}

As in Theorem~\ref{thm:T1}, we group permutations of the same combination $i_1,\ldots i_n$ together -- under such a permutation $\sigma \in S_n$, the determinant of the submatrix of $A$ remains unchanged up to multiplication by the sign of the permutation $\sgn(\sigma)$:

\begin{align*}
\det {A^T}B 
& = \sum_{i_1 <  \ldots i_n} \sum_\sigma 
B_{i_{\sigma (1)}}^1 \ldots B_{i_{\sigma(n)}}^n 
\det A_{i_{\sigma(1)},\ldots i_{\sigma(n)}} \\
& = \sum_{i_1 <  \ldots i_n} \det A_{i_1,\ldots i_n} \sum_\sigma 
\sgn (\sigma) B_{i_{\sigma (1)}}^1 \ldots B_{i_{\sigma (n)}}^n \\
& = \sum_{i_1 < \ldots i_n} \det A_{i_1,\ldots i_n} \det B_{i_1,\ldots i_n}
\end{align*}

Which is precisely the dot product of $\vdet A$ and $\vdet B$.

\end{proof}

\subsection{Discussion}\label{sec:vdet-disc}
Theorem~\ref{thm:T2} provides a crucial insight into the inner product structure on the space of vector determinants in $\mathbb{R}^{\binom{m}{n}}$. In particular, it resolves the following two facts as special cases: (1) letting $A$ and $B$ be square, $\det(AB) = \det(A) \det(B)$, (2) letting $A$ and $B$ be equal, $\detl^2 A = |\vdet A|^2$ (the latter provides a geometric interpretation to the magnitude of the vector determinant).

A way to think about the vector determinant is that it is a ``vector representation'' to $\bigwedge_j A_j$, similar to how a cross product is a vector representation to the wedge product in three dimensions. In a sense, the vector determinant may be viewed as a ``signed'' version of the determinant-like function, but with a direction rather than a sign. For example, one may restate result Eq.~\eqref{detl-sol} as follows -- if the equation $Ax = b$ is solvable, its solutions $x_j$ can be written as the ratio between proportional vector determinants:

\begin{equation}\label{vdet-sol}
x_j \vdet A = \vdet (A_1,\ldots A_j,b,A_{j+1},\ldots A_n)
\end{equation}

If the system is not solvable, the vector determinants are not proportional, and Eq.~\eqref{vdet-sol} has no solution for any $j$.

\section{The g-determinant}\label{sec:gdet}

We provide an inductive proof of Eq.~\eqref{gdet} from the defining base case Eq.~\eqref{gdet-base} and Laplace recursion Eq.~\eqref{gdet-ih}.

\begin{thm}\label{thm:T3}
The following generalised determinant is the unique function on $m$ by $n$ satisfying Eq.~\eqref{gdet-base} for $m$ by 1 matrices and Laplace's expansion along the first column:

\begin{equation*}
\gdet A = \sum_{i_1<\ldots i_n}(-1)^{\sum_j i_j + j}\det A_{i_1,\ldots i_n}
\end{equation*}
\end{thm}

\begin{proof}
The base case is clear. It suffices (by diagonal induction from the $n = 1$ line) to show that the statement for $(m - 1) \times (n - 1)$ matrices implies the statement for $m \times n$ matrices. Expanding the $(m-1)\times (n -1)$ g-determinants in the Laplace expansion,

\begin{align*}
\gdet A 
&= \sum_{i_1 = 1}^m (-1)^{i_1 + 1} a_{i_1 1} 
\sum_{i_2 < \ldots i_n \ne i_1} (-1)^{\sum\limits_{j = 2}^n (i_j + j)} \det A_{i_2,\ldots i_n}^{2,\ldots n} \\
&= \sum_{i_1 = 1}^m \sum_{i_2 < \ldots i_n \ne i_1} (-1)^{\sum\limits_{j = 1}^n (i_j + j)} a_{i_1 1} \det A_{i_2,\ldots i_n}^{2,\ldots n}
\end{align*}

We group all permutations of a given set $i_1, i_2,... i_n$ together, allowing us to write each term with ascending $i_j$ -- the determinant terms in each such case are multiplied by $(-1)^{k+1}$ in bringing the rows of the submatrix to their natural order:

\begin{align*}
\gdet A 
&= \sum_{i_1 < \ldots i_n} (-1)^{\sum_j i_j + j} \sum_{k = 1}^n 
(-1)^{k + 1} a_{i_k 1} \det A_{\{i_{j\ne k}\}}^{\{k\}^C} \\
&= \sum_{i_1 < \ldots i_n} (-1)^{\sum_j i_j + j} \det A_{i_1,...i_n}
\end{align*}
Uniqueness follows from the well-foundedness of the recursion.
\end{proof}

\subsection{Discussion}

It's worth noting that in order to preserve parallels with the other generalised determinants, we have used the opposite convention to \cite{Radic1} with regards to the rows and columns of the matrix -- our g-determinant is the g-determinant of the transpose as per the original convention.

Although the g-determinant does not have a conventional geometric interpretation in terms of scaling of volumes under a non-square transformation, the $m$ by 2 g-determinant does have a rather unique geometric application in relation to the areas of polygons -- explicitly, the area of an $m$-vertex polygon in $\mathbb{R}^2$ whose vertices are given by the column vectors $A_i$ is given by

\begin{equation}\label{gdet-poly}
[A_1,...A_m] = \frac12 \gdet (A_1 + A_2, A_2 + A_3, \ldots A_m + A_1)^T
\end{equation}

We state the above result without proof, as it follows in a straightforward fashion from the shoelace formula \cite{Radic2}. Notably, it is easy to show that in the continuous limit with $m \to \infty$, Eq.~\eqref{gdet-poly} reduces to the standard expression for the area contained within a curve:

\begin{equation}\label{gdet-curve}
[C] = \frac12 \int_C x\, dy - y \, dx
\end{equation}

The proof does not differ much from a standard continuous generalisation of the shoelace formula \cite{NOSEAL}, and is left as an exercise to the reader.

\section{Conclusion}

We have discussed the properties of three significant tall-matrix determinants, as well as some insight into their geometric interpretations. Our treatment is by no means comprehensive -- notably, we did not discuss the generalisation in \cite{Joshi}, with the exception of a brief statement of its definition in Eq.~\eqref{tdet}, as there is relatively less literature on this generalisation, and its known properties are mostly elementary.

The g-determinant is perhaps the most well-studied of the generalised determinants, and we have only covered a selection of its properties, including of its geometric properties. An interested reader might wish to consult references \cite{Radic1} \cite{Radic2} \cite{Radic3} for a deeper look at the research in this area.

Although we have covered the vector determinant as a vector in $\mathbb{R}^{\binom{m}{n}}$ in agreement with the literature, it may be more revealing to consider it simply as an $n$-vector in $\mathbb{R}^m$. All its discussed properties of the generalisation are preserved under this interpretation -- additionally, one obtains a rather natural way to think about the signs of each component determinant as the signs of the projections of the determinant multivector onto basis $n$-forms. The projection (and its sign) onto an arbitrary $n$-dimensional subspace can be obtained via an inner product with a unit $n$-form on that subspace -- this ``directional determinant'' can be said to be analogous to directional derivatives in analysis.

It is not clear if the rather similar roles of the determinants of the square submatrices in the various generalised determinants may offer any further insight, or if there is any fundamental relation between the g-determinant, the vector determinant and the determinant in \cite{Joshi}.



\begin{thebibliography}{9}

\bibitem{Pasu1}
A. Pallavi Sudhir,
\emph{Defining the determinant-like function for $m$ by $n$ matrices using the exterior algebra},
Advances in Applied Clifford Algebras.
\textbf{23(4)} (2013), 787-792.

\bibitem{Pasu2}
A. Pallavi Sudhir, 
\emph{On the determinant-like function and the vector determinant},
Advances in Applied Clifford Algebras.
\textbf{24(3)} (2014). 805-807.

\bibitem{Pasu3}
A. Pallavi Sudhir,
\emph{On the properties of the determinant-like function},
International Conference on Mathematical Sciences.
Chennai (2014).

\bibitem{Pyle}
H. Pyle,
\emph{Non-square determinants and multilinear Vectors},
Mathematics Association of America.
\textbf{35(2)} (1962), 65-69.

\bibitem{Radic1}
M. Radic,
\emph{A definition of determinant of rectangular matrix}, 
Glas. Mat.
\textbf{1(21)} (1966), 17-22.

\bibitem{Radic2}
M. Radic,
\emph{About the determinant of a 2 by $n$ matrix and its geometric interpretation}, 
Contributions to Algebra and Geometry.
\textbf{46(1)} (2005), 321-349.

\bibitem{Radic3}
R. Susanj \& M. Radic,
\emph{Geometric meaning of one generalisation of the determinant of square matrix}, 
Glas. Mat. III. Ser.
\textbf{29(2)} (1994), 217-233.

\bibitem{Joshi}
V. N. Joshi,
\emph{A determinant for rectangular matrices},
Bull. Austral. Math. Soc.
\textbf{21} (1980), 137-146

\bibitem{NOSEAL}
Retrieved from Mathematics Stack Exchange Community Blog, 
\emph{Green's theorem and area of polygons}.
Posted by user ``apnorton'', (2014). \href{https://math.blogoverflow.com/2014/06/04/greens-theorem-and-area-of-polygons}{math.blogoverflow.com/2014/06/04/greens-theorem-and-area-of-polygons}

\end{thebibliography}
\end{document}